\documentclass[10pt]{amsart}
\usepackage{amssymb,amstext,amsmath,amscd,amsthm,amsfonts,enumerate,latexsym,stmaryrd,multicol,geometry,graphicx,mathrsfs}
\usepackage{hyperref}
\usepackage[usenames]{color}
\usepackage[all]{xy}
\newtheorem{thm}{Theorem}[section]
\newtheorem{lem}[thm]{Lemma}
\newtheorem{prop}[thm]{Proposition}
\newtheorem{cor}[thm]{Corollary}
\theoremstyle{definition}
\newtheorem{dfn}[thm]{Definition}
\newtheorem{ques}[thm]{Question}

\newtheorem{rem}[thm]{Remark}

\theoremstyle{remark}

\newtheorem*{claim*}{Claim}

\newtheorem*{conv}{Convention}

\numberwithin{equation}{thm}
\def\ABdim{\operatorname{AB\text{-}dim}}
\def\ac{\operatorname{\mathtt{AC}}}

\def\cm{\operatorname{\mathtt{CM}}}
\def\CIdim{\operatorname{CI\text{-}dim}}
\def\depth{\operatorname{depth}}
\def\dtrc{\operatorname{\mathtt{DTRC}}}
\def\Ext{\operatorname{Ext}}
\def\fab{\operatorname{\mathtt{fab}}}
\def\fcx{\operatorname{\mathtt{fcx}}}
\def\fgd{\operatorname{\mathtt{fgd}}}
\def\fpd{\operatorname{\mathtt{fpd}}}
\def\G{\mathcal{G}}
\def\garc{\operatorname{\mathtt{GARC}}}
\def\gd{\operatorname{Gdim}}
\def\ge{\geqslant}
\def\geq{\geqslant}
\def\Gdim{\operatorname{G\text{-}dim}}
\def\grade{\operatorname{grade}}
\def\hExt{\operatorname{Ext}_R^{\gg0}}
\def\Hom{\operatorname{Hom}}

\def\I{\mathbb{I}}
\def\le{\leqslant}
\def\leq{\leqslant}

\def\mod{\operatorname{mod}}

\def\P{\mathbb{P}}
\def\p{\mathfrak{p}}
\def\pd{\operatorname{pd}}

\def\red{\operatorname{red}}

\def\ri{\operatorname{red-\mathbb{I}}}

\def\rABdim{\operatorname{red\text{-}AB\text{-}dim}}

\def\rCIdim{\operatorname{red\text{-}CI\text{-}dim}}
\def\rGdim{\operatorname{red\text{-}G\text{-}dim}}

\def\rpd{\operatorname{red\text{-}pd}}

\def\sred{\operatorname{{}^\ast red}}
\def\sri{\operatorname{{}^\ast red-\mathbb{I}}}

\def\Supp{\operatorname{Supp}}
\def\syz{\Omega}
\def\Tor{\operatorname{Tor}}
\def\tr{\operatorname{Tr}}
\def\ul{\operatorname{\mathtt{Ul}}}
\def\uac{\operatorname{\mathtt{UAC}}}
\def\X{\mathcal{X}}
\def\Y{\mathcal{Y}}
\def\Z{\mathbb{Z}}
\begin{document}
\title{Homological properties and finiteness of reducing invariants}
\author{Tokuji Araya}
\address[T.A.]{Department of Applied Science, Faculty of Science, Okayama University of Science, Ridaicho, Kitaku, Okayama 700-0005, Japan.}
\email{araya@ous.ac.jp}
\author{Naoya Hiramatsu}
\address[N.H.]{Institute for the Advancement of Higher Education, Okayama University of Science, 1-1 Ridai-cho, Kita-ku, Okayama 700-0005, Japan}
\email{n-hiramatsu@ous.ac.jp} 
\author{Ryo Takahashi}
\address[R.T.]{Graduate School of Mathematics, Nagoya University, Furocho, Chikusaku, Nagoya 464-8602, Japan}
\email{takahashi@math.nagoya-u.ac.jp}
\urladdr{https://www.math.nagoya-u.ac.jp/~takahashi/}
\thanks{N.H. was partly supported by JSPS KAKENHI Grant Number 25K06966. 
R.T. was partly supported by JSPS Grant-in-Aid for Scientific Research 23K03070}
\dedicatory{Dedicated to the memory of Professor Mitsuyasu Hashimoto}
\begin{abstract}
We study reducing invariants of modules related to certain homological properties. 
For modules of finite reducing projective dimension, we establish grade inequalities. 
We prove that if $\P$ is the (uniform) Auslander condition, or the generalized Auslander--Reiten conjecture, or dependence of the total reflexivity conditions, then a module satisfies $\P$ provided that it has finite reducing invariant with respect to $\P$.
\end{abstract}
\maketitle
\section{Introduction}\label{A}

Throughout the present paper, let $R$ be a commutative noetherian ring.
All $R$-modules are assumed to be finitely generated. 
The study of homological invariants of modules has been playing a central role in commutative algebra. 
Classical invariants, such as projective dimension, Gorenstein dimension, and complete intersection dimension, have been intensively investigated in connection with grade (in)equalities, vanishing of Ext and Tor, and the structures of Cohen--Macaulay rings and of Gorenstein rings.

Araya and Celikbas \cite{AC} introduced the notion of a reducing invariant, which is inspired by the notion of reducible complexity introduced by Bergh \cite{B}. 
It measures the extent to which a given module can be reduced, via repeated applications of syzygies or of certain short exact sequences, to a module having a desirable property, such as finite homological dimension. 
It has been studied by many authors \cite{red, CDKM22, CKLM24}.
To be more precise, let $\I_R$ be an {\em invariant} of $R$-modules, i.e., a map from the set of isomorphism classes of $R$-modules to the set $\Z\cup\{\pm\infty\}$.
Then for an $R$-module $M$ we write $\ri_R(M)<\infty$ if there exists a family
$$
\{0\to K_{i-1}^{\oplus a_i}\to K_i\to\syz^{n_i}K_{i-1}^{\oplus b_i}\to0\}_{i=1}^r
$$
of short exact sequences of $R$-modules with $r,n_i\ge0$,  $a_i,b_i>0$, $K_0=M$ and $\I_R(K_r)<\infty$.

Our first main result establishes a grade inequality for modules of finite reducing projective dimension.

\begin{thm}[Theorem \ref{B5}]\label{A1}
Let $R$ be a Cohen--Macaulay local ring.
Let $L,M,N$ be $R$-modules.
Suppose that $\rpd_RN<\infty$ and $\Tor_{>0}^R(M,N)=0$.
\begin{enumerate}[\rm(1)]
\item The inequality $\grade(L,M) \leq \grade(L, M \otimes N) + \depth R - \depth N$ holds.

\item Suppose that $\Supp L \subseteq \Supp N$. Then the inequality $\grade(L, M \otimes N) \leq \grade(L, N)$ holds. 
\end{enumerate}
\end{thm}

We also investigate the relationships between reducing invariants and several homological properties.
Recall that an $R$-module $M$ is said to satisfy:
\begin{itemize}
\item
the {\em Auslander condition}, provided that there exists an integer $b_M\ge0$ such that for every $R$-module $N$, if $\Ext_R^{\gg0}(M,N)=0$, then $\Ext_R^{>b_M}(M,N)=0$.
\item
the {\em uniform Auslander condition}, provided that there exists an integer $b_M\ge0$ such that for every nonzero $R$-module $N$, if $\Ext_R^{\gg0}(M,N)=0$, then $\Ext_R^{>b_M}(M,N)=0$ and $\Ext_R^{b_M}(M,N)\neq0$.
\item
the {\em generalized Auslander--Reiten conjecture}, provided that if there exists an integer $n\ge0$ such that $\Ext_R^{>n}(M,M\oplus R)=0$, then $M$ has projective dimension at most $n$.
\item
{\em dependence of the total reflexivity conditions}, provided if $\Ext_R^{>0}(M,R)=0$, then $M$ is totally reflexive.
\end{itemize}

Now we are able to state the second main result of this paper.
This theorem says that, for instance, when we want to show that every $R$-module $M$ with $\I_R(M)<\infty$ satisfies dependence of total reflexivity conditions, we have only to verify that every $R$-module $M$ with $\ri_R(M)<\infty$ satisfies that property.

\begin{thm}[Corollary \ref{13}]\label{5}
Let $\I_R$ be an invariant of $R$-modules.
Consider the properties of modules:
\begin{enumerate}[\qquad\rm(i)]
\item
being maximal Cohen--Macaulay,
\item
being Ulrich,
\item
the Auslander condition, 
\item 
the uniform Auslander condition,
\item
the generalized Auslander--Reiten conjecture, 
\item
dependence of the total reflexivity conditions.
\end{enumerate}\begin{enumerate}[\rm(1)]
\item
Let $\P$ be one of the properties {\rm(iii)}, {\rm(iv)}, {\rm(v)} and {\rm(vi)}.
Assume that any $R$-module $M$ such that $\I_R(M)<\infty$ satisfies $\P$.
Then any $R$-module $M$ such that $\ri_R(M)<\infty$ satisfies $\P$ as well.
\item
Suppose that $R$ is a Cohen--Macaulay local ring.
Let $\P$ be either of the properties {\rm(i)} and {\rm(ii)}.
If all $R$-modules $M$ with $\I_R(M)<\infty$ satisfy $\P$, then so are all $R$-modules $M$ with $\ri_R(M)<\infty$.
\end{enumerate}
\end{thm}

Applying the above theorem, we provide some classes of $R$-modules which satisfy the conditions {\rm(iii)}, {\rm(iv)}, {\rm(v)} and {\rm(vi)} in Theorem \ref{5}. 
We obtain the following result.

\begin{cor}[Corollary \ref{E3}]
Let $M$ be an $R$-module which has reducible complexity. 
Then $M$ satisfies the (uniform) Auslander condition, the (generalized) Auslander--Reiten conjecture and dependence of the total reflexivity conditions.
\end{cor}

The organization of the paper is as follows. 
In Section \ref{B}, we introduce a modified version of reducing invariants (Definition \ref{maindef}) and prove Theorem \ref{A1}.
In Section \ref{C}, we define a closure operator on subcategories of $\mod R$ with respect to reducing invariants (Definition \ref{maindef2}), and in Section \ref{D}, we discuss the properties of this operator for subcategories related to projective or Gorenstein dimension.
Section \ref{E} is devoted to proving Theorem \ref{5} and providing classes of modules to which the theorem applies (Corollaries \ref{E1}, \ref{E2}, and \ref{E3}).
\begin{conv}
Throughout, let $R$ be a commutative noetherian ring.
Denote by $\mod R$ the category of finitely generated $R$-modules, and by $\mod R/\cong$ the set of isomorphism classes of finitely generated $R$-modules.
We assume that all $R$-modules are finitely generated, and that all subcategories of $\mod R$ are strictly full.
Hence, there is a one-to-one correspondence between the subcategories of $\mod R$ and the subsets of $\mod R/\cong$.
\end{conv}
\section{Grade inequality}\label{B}

The notion of a reducing invariant has been introduced by Araya and Celikbas \cite{AC}. 
We adopt the modified version in this paper.

\begin{dfn}\label{maindef}
Let $\I_R :(\mod R/\cong)\to\Z\cup\{\pm\infty\}$ be a map.
Let $M$ be an $R$-module.
We define the {\em reducing $\I_R$-invariant} of $M$, denoted $\ri_RM$, as follows.
\begin{enumerate}[(1)]
\item
One has $\I_R(M)<\infty$ if and only if $\ri_R(M)=0$.
\item
Suppose that $\I_R(M)=\infty$.
If there exists a family $\{0\to K_{i-1}^{\oplus a_i}\to K_i\to\syz^{n_i}K_{i-1}^{\oplus b_i}\to0\}_{i=1}^r$ of exact sequences of $R$-modules such that $n_i\ge0$, $r,a_i,b_i>0$, $K_0=M$ and $\I_R(K_r)<\infty$, then $\ri_RM$ is the minimum of such integers $r$.
\item
If no such exact sequences as in (2) exist, then $\ri_RM=\infty$.
\end{enumerate}
\end{dfn}

It is nothing but $\ri_R M$ for the reducing invariants of ${\mathbb{I}}_{R} = \{ \pd , \CIdim, \ABdim, \Gdim  \}$ are the same as those introduced in \cite{AC} with respect to a homological invariant of $R$-modules.

A fundamental problem in homological algebra is to determine when the depth formula holds.

\begin{dfn}
We say that the \emph{depth formula} holds for $R$-modules $M$ and $N$ if there is an equality
\begin{equation}\label{depth formula}
\depth M + \depth N= \depth R + \depth(M \otimes N)  
\end{equation}
\end{dfn}

In the context of reducing invariants, the following theorem appears in \cite{CKLM24}.
Let $M$ and $N$ be $R$-modules. 
We define the integer as follows: 
$$
q^R(M, N) = \sup \{ i | \Tor _i ^ R (M, N)\not= 0 \}.   
$$

\begin{thm}\label{DF}\cite[Theorem 1.2]{CKLM24}
Let $R$ be a Cohen--Macaulay local ring and let $M$ and $N$ be $R$-modules with $q^R(M, N)=0$. 
If $\rpd N < \infty$ then the depth formula holds for $M$ and $N$. 
\end{thm}

For $R$-modules $L$ and $M$, we define $\grade (L, M)$ by $\inf \{ i | \Ext _R ^ i (L, M) \not= 0 \}$. 
Here we define $\grade (L, 0) = \infty$. 
We remark that $\grade (L, M) = \inf \{ \depth M_{\p} | \p \in \Supp L\}$ (cf. \cite[Proposition 1.2.10]{BH}).

\begin{thm}\label{GI}\cite[Theorem 3.1]{AY98}
Let $M$ and $N$ be $R$-modules and suppose that $\pd N < \infty$ and $q^R (M, N) = 0$. 
Then, for any $R$-module $L$, the following inequality holds: 
\begin{equation}\label{grade inequality}
\grade (L, M \otimes N ) \leq \grade (L, M) \leq \grade (L, M \otimes N) + \pd N.  
\end{equation}
\end{thm}

The inequality (\ref{grade inequality}) is called {\it grade inequality}. 
In this section, we investigate the reducing projective dimension version of Theorem \ref{GI}.  
Actually, we show the following theorem.

\begin{thm}\label{B5}
Let $R$ be a Cohen--Macaulay local ring and $L$, $M$, $N$ be $R$-modules. 
Suppose that $\rpd N < \infty$ and $q^R (M,N) =0$. 
\begin{enumerate}[\rm(1)]
\item The inequality $\grade (L, M) \leq \grade (L, M \otimes N ) + \depth R - \depth N$ holds. 
\item Suppose that $\Supp L \subseteq \Supp N$ then the inequality $\grade (L , M \otimes N) \leq \grade (L, N)$ holds. 
\end{enumerate}
\end{thm}

To prove the theorem, we state some lemmas. 
 
\begin{rem}\label{B1}\cite[Chapter 3. Exercise 19.]{AT69}
Let $M$ and $N$ be $R$-modules, 
\begin{enumerate}[\rm(1)]
\item $\Supp M^{\oplus a} = \Supp M$. 
\item If $N$ is a submodule of $M$, $\Supp N \subseteq \Supp M$. 
\item $\Supp M \otimes N = \Supp M \cap \Supp N$. 
\end{enumerate}
\end{rem}

In the rest of this secation, $K_r$ denotes the module which appears in the definition of $\rpd N=r<\infty$. 
Note that $\pd K_r <\infty$.

\begin{lem}\label{B2}
If $\rpd N < \infty$, then $\Supp N \subseteq \Supp K_r$. 
\end{lem}

\begin{proof}
Since $\rpd N < \infty$, by the definition, we have 
$$
0 \to K_{i-1} ^{\oplus a_i} \to K_{i} \to \Omega ^{n_i} K_{i-1}^{\oplus b_i} \to 0
$$
for $1\le i\le r$, where $K_0 = N$.  
By Remark \ref{B1}(1)(2), $\Supp K_{i-1} \subseteq \Supp K_{i}$, so that $\Supp N \subseteq \Supp K_r$. 
\end{proof}

\begin{thm}\label{B3}
Let $R$ ba a local ring. 
Suppose that  $\depth N \leq \depth R$ for an $R$-module $N$. 
If $\rpd N < \infty$ then $\depth N = \depth {K_r}$. 
Consequently, $\grade (L, K_r) \leq \grade (L, N)$ for an $R$-module $L$. 
\end{thm}

\begin{proof}
First, we remark that, by the assumption and using the depth lemma, we have $\depth \Omega ^n M \geq \min \{ \depth R, \depth M + n\}$ for every $R$-module $M$. 
Particularly, $\depth \Omega ^n M \geq \depth M$. 
Since $\rpd N = r < \infty$, there are short exact sequences
$$
0 \to K_{i-1} ^{\oplus a_i} \to K_{i} \to \Omega ^{n_i} K_{i-1}^{\oplus b_i} \to 0
$$
for $i = 1, ..., r$. 
By using the depth lemma, $\depth K_i \geq \min \{ \depth K_{i-1}, \depth  \Omega ^{n_i} K_{i-1} \}$
Since $\depth \Omega ^{n_i}K_{i-1} \geq \depth K_{i-1}$, we have $\depth K_i \geq \depth K_{i-1}$. 
Similarly we have $\depth K_{i-1} \geq \min \{ \depth K_{i}, \depth  \Omega ^{n_i} K_{i} +1 \} \geq \depth K_{i}$. 
Hence $\depth K_{i-1} = \depth K_{i}$ for all $i$. 
So, $\depth N = \depth K_r$ holds. 

Let $L$ be an $R$-module and $\p$ in $\Supp L$. 
According to the observation above, $\depth N_{\p} = \depth {K_r}_{\p}$ holds if $\p \in \Supp N$. 
By Lemma \ref{B2}, we obtain the equality $\grade (L, N) = \grade (L, K_r)$. 
Since we may have the case that $\p \in \Supp K_r \backslash \Supp N$, the inequality may happen. 
\end{proof}

\begin{prop}\label{B4}
Let $R$ be a Cohen--Macaulay local ring and let $M$ and $N$ be $R$-modules with $q^R (M,N) =0$. 
Suppose that $\rpd N < \infty$. 
Then $\depth _{R_{\p}} (M \otimes N)_{\p} = \depth_{R_{\p}} (M \otimes K_r)_{\p}$ for each $\p \in \Supp M \otimes N$. 
\end{prop}

\begin{proof}
We put $K = K_r$. 
Let $\p \in \Supp M \otimes N$. 
By Lemma \ref{B2} and Remark \ref{B1}(3), we have $\p \in \Supp M \otimes K$. 
Bear in mind that $\rpd_{R_{\p}} N_{\p} < \infty$ (see \cite[Lemma 2.5]{CKLM24}) and $q^{R_{\p}} (M_{\p}, N_{\p}) = 0$. 
Moreover, $\pd_{R_{\p}} K_{\p } < \infty$ and $q^{R_{\p}} (M_{\p}, K_{\p}) = 0$. 
Thus we can apply the depth formula (Theorem \ref{DF}) for $M_{\p}$ and $N_{\p}$ and for $M_{\p}$ and $K_{\p}$ respectively. 
That is,
\begin{align*}
&\depth_{R_{\p}} M_{\p} + \depth_{R_{\p}} N_{\p }= \depth_{R_{\p}} M_{\p} \otimes N_{\p} + \depth_{R_{\p}} R_{\p},\\
&\depth_{R_{\p}} M_{\p} + \depth_{R_{\p}} K_{\p }= \depth_{R_{\p}} M_{\p} \otimes K_{\p} + \depth_{R_{\p}} R_{\p}.
\end{align*}
As $\depth_{R_{\p}} N_{\p} = \depth_{R_{\p}} K_{\p}$ by Theorem \ref{B3}, we get the equality $\depth_{R_{\p}} M_{\p} \otimes N_{\p} = \depth_{R_{\p}} M_{\p} \otimes K_{\p}$. 
\end{proof}

\begin{proof}[Proof of Theorem \ref{B5}]
For simplicity, we put $K = K_r$. 
Let $\p \in \Supp L$. 
Suppose that $(M \otimes N)_{\p}=0$, then $\depth _{R_{\p}}(M \otimes N)_{\p} = \infty$. 
So the inequality $\depth M_{\p} \leq \depth (M \otimes N)_{\p}$ holds. 
We assume that $(M \otimes N)_{\p}\not=0$. 
Since $\Supp M \otimes N \subseteq \Supp M \otimes K$, we have $(M \otimes K)_{\p}\not=0$. 
Particularly, $K_{\p} \not=0$. 
Notice that $\pd K$ is finite and $\pd_{R_{\p}} K_{\p}$ is so. 
By the Auslander-Buchsbaum formula, $\pd K_{\p} = \depth R_{\p} - \depth K_{\p}$.  
Apply the depth formula for $K_{\p}$, one has
$$
\begin{array}{ll}
\depth M_{\p} &= \depth (M \otimes K)_{\p} + \depth R_{\p} - \depth K_{\p} \\
&= \depth (M \otimes K)_{\p} + \pd_{R_{\p}} K_{\p}. \\
\end{array}
$$
By virtue of Lemma \ref{B3} and Proposition \ref{B4}, we obtain $\depth M_{\p} = \depth (M \otimes N)_{\p} + \pd_{R_{\p}} K_{\p}$ for each $\p \in \Supp L \cap \Supp M \otimes N$. 
Use the Auslander-Buchsbaum formula again, $\pd_{R_{\p}}K_{\p} \leq \pd K = \depth R - \depth K = \depth R - \depth N.$
Thus we have 
$$
\grade (L, M) \leq \grade (L, M\otimes N) +\pd K = \grade (L, M\otimes N) + \depth R - \depth N. 
$$

Now we show (2). 
First, we claim that $\Supp M \otimes N \cap \Supp L = \Supp M \otimes K \cap \Supp L$. 
By the assumption and Lemma \ref{B2}, $\Supp L \subseteq \Supp N \subseteq \Supp K$. 
Thus, 
$$
\begin{array}{rl}
\Supp M \otimes N \cap \Supp L &= \Supp M \cap  \Supp N \cap \Supp L \\
&= \Supp M \cap \Supp L  \\
&= \Supp M \cap \Supp K \cap \Supp L . \\
&= \Supp M \otimes K \cap \Supp L . \\
\end{array}
$$
Combining the equality and Proposition \ref{B4}, one can show $\grade (L, M \otimes N) = \grade (L, M \otimes K)$.  
Since $\pd K$ is finite, we have $\grade (L, M \otimes N) = \grade (L, M \otimes K) \leq \grade (L, K) \leq \grade (L, N)$ by grade inequality (Theorem \ref{GI}) for $K$ and Theorem \ref{B3}. 
\end{proof}

\begin{rem}\label{B7}
The assumption $\Supp L \subseteq \Supp N$ of Theorem \ref{B5} (2) cannot be omitted. 
Let $R=k[\![x, y]\!]/(xy)$, $L=R/(x)$, $M=R$ and $N=R/(y)$, where $k$ is a field. 
Then $\Supp L \not\subseteq \Supp N$. 
We have a short exact sequence $0 \to N \to R \to L \to 0$, so that $\rpd N = 1$. 
Since $\Hom_R (L,N)=0$ and $\Ext^1_R (L,N) \not=0$, we have $\grade (L, N) = 1$. 
We can also show that $\Ext^i_R (L, M) =0$ for $i > 0$ since $R$ is Gorenstein and $L$ is a maximal Cohen--Macaulay $R$-module. 
Hence $\grade (L, M) = 0$. 
Summing up with $\grade (L, N) = 1$ and $\grade (L, M) = 0$, we see that 
$$
\grade (L, M \otimes N)=\grade (L, R \otimes N) = 1 \not \leq \grade (L, M) = 0. 
$$
This example also tells us that the grade inequality (Theorem \ref{GI}) does not hold even if an $R$-module $N$ has a finite complete intersection dimension. 
As $R$ is complete intersection, $\CIdim N < \infty$ (\cite[Theorem 1.3]{AGP97}).  
\end{rem}

\section{Reducing closures}\label{C}

In this section, we introduce subcategories of $\mod R$ in terms of reducing invariants. 
These subcategories induce a closure operator on $\mod R/\cong$. 
We also give the condition that these subcategories are closed. 

\begin{dfn}\label{maindef2}
For a subcategory $\X$ of $\mod R$ we define the subcategories $\red\X$ and $\sred\X$ of $\mod R$ by:
\begin{align*}
\red\X&=\left\{\begin{array}{r|l}
\!\!\!M\in\mod R&
\begin{matrix}
\text{there exists a family of exact sequences in $\mod R$}\\
\{0\to K_{i-1}^{\oplus a_i}\to K_i\to\syz^{n_i}K_{i-1}^{\oplus b_i}\to0\}_{i=1}^r\\
\text{with $r,n_i\ge0$, $a_i,b_i>0$, $K_0=M$ and $K_r\in\X$}\end{matrix}
\end{array}\!\right\}\!,\\
\sred\X&=\left\{\begin{array}{r|l}
\!\!\!M\in\mod R&
\begin{matrix}
\text{there exists a family of exact sequences in $\mod R$}\\
\{0\to K_{i-1}\to K_i\to\syz^{n_i}K_{i-1}\to0\}_{i=1}^r\\
\text{with $r,n_i\ge0$, $K_0=M$ and $K_r\in\X$}\end{matrix}
\end{array}\!\right\}\!,
\end{align*}
We call $\red\X$ and $\sred\X$ the {\em reducing closure} and the {\em ${}^\ast$reducing closure} of $\X$, respectively.
Note by definition that there is always an inclusion $\red\X\supseteq\sred\X$.
\end{dfn}

\begin{rem}
The subcategory $\sred\X$ comes from another modification of reducing invariants by Araya and Takahashi \cite[Definition 3.1]{red}. 
See also Definition \ref{D1}. 
\end{rem}

\begin{dfn}
Let $X$ be a set.
A {\em Kuratowski closure operator} on $X$ is defined to be a map $f:2^X\to2^X$ which satisfies the following four properties.
\begin{enumerate}
\item[\rm(K1)]
One has $f(\emptyset)=\emptyset$.
\item[\rm(K2)]
For a subset $A$ of $X$ one has $A\subseteq f(A)$.
\item[\rm(K3)]
For a subset $A$ of $X$ one has $f(A)=f(f(A))$.
\item[\rm(K4)]
For two subsets $A,B$ of $X$ one has $f(A\cup B)=f(A)\cup f(B)$.
\end{enumerate}
\end{dfn}

\begin{thm}\label{3}
The assignments $\red(-)$ and $\sred(-)$ define Kuratowski closure operators on $\mod R/\cong$.
\end{thm}

\begin{proof}
We only prove the assertion on $\red(-)$; that on $\sred(-)$ can be shown in a similar way.

(K1) It is evident that $\red\emptyset=\emptyset$.

(K2) Let $\X$ be a subcategory of $\mod R$.
Letting $r=0$ in the definition of $\red\X$ shows that $\X\subseteq\red\X$.

(K3) Let $\X$ be a subcategory of $\mod R$.
By (K2), we get $\red\X\subseteq\red(\red\X)$.
Let $M\in\red(\red\X)$.
There is a family $\{0\to K_{i-1}^{\oplus a_i}\to K_i\to\syz^{n_i}K_{i-1}^{\oplus b_i}\to0\}_{i=1}^r$ of exact sequences with $r,n_i\ge0$, $a_i,b_i>0$, $K_0=M$ and $K_r\in\red\X$.
Hence there is a family $\{0\to K_{i-1}^{\oplus a_i}\to K_i\to\syz^{n_i}K_{i-1}^{\oplus b_i}\to0\}_{i=r+1}^s$ of exact sequences with $s\ge r$, $n_i\ge0$, $a_i,b_i>0$ and $K_s\in\X$.
The family $\{0\to K_{i-1}^{\oplus a_i}\to K_i\to\syz^{n_i}K_{i-1}^{\oplus b_i}\to0\}_{i=1}^s$ shows that $M\in\red\X$.

(K4) Let $\X,\Y$ be subcategories of $\mod R$.
Take $M\in\red(\X\cup\Y)$.
There is a family $\{0\to K_{i-1}^{\oplus a_i}\to K_i\to\syz^{n_i}K_{i-1}^{\oplus b_i}\to0\}_{i=1}^r$ of exact sequences with $r,n_i\ge0$, $a_i,b_i>0$, $K_0=M$ and $K_r\in\X\cup\Y$.
Hence $K_r$ is in either $\X$ or $\Y$.
If $K_r$ is in $\X$ (resp. $\Y$), then $M$ is in $\red\X$ (resp. $\red\Y$).
Thus $\red(\X\cup\Y)\subseteq\red\X\cup\red\Y$.

Pick any $M\in\red\X\cup\red\Y$.
Then $M$ is in either $\red\X$ or $\red\Y$.
In the former (resp. latter) case, there is a family $\{0\to K_{i-1}^{\oplus a_i}\to K_i\to\syz^{n_i}K_{i-1}^{\oplus b_i}\to0\}_{i=1}^r$ of exact sequences such that $r,n_i\ge0$, $a_i,b_i>0$, $K_0=M$ and $K_r$ is in $\X$ (resp. $\Y$).
In either case $K_r$ is in $\X\cup\Y$, and thus $\red(\X\cup\Y)$ contains $\red\X\cup\red\Y$.
\end{proof}

\begin{ques}
Can we say anything about the topology defined by the Kuratowski closure operator $\red(-)$?
\end{ques}

\begin{dfn}
Let $\X$ be a subcategory of $\mod R$.
We say that $\X$ is {\em reducingly closed} (resp. {\em ${}^\ast$reducingly closed}) provided that $\X=\red\X$ (resp. $\X=\sred\X$).
According to Theorem \ref{3}, it always holds that $\X\subseteq\sred\X\subseteq\red\X$.
In particular, if $\X$ is reducingly closed, then it is ${}^\ast$reducingly closed as well.
\end{dfn}

The following lemma will be useful to show that a given subcategory is reducingly closed.

\begin{lem}\label{1}
Let $\X$ be a subcategory of $\mod R$.
The following two conditions are equivalent.
\begin{enumerate}[\rm(1)]
\item
The subcategory $\X$ is reducingly closed.
\item
For every short exact sequence of the form $0\to M^{\oplus a}\to K\to\syz^nM^{\oplus b}\to0$ with $n\ge0$ and $a,b>0$, one has that if $K$ belongs to $\X$, then so does $M$.
\end{enumerate}
\end{lem}

\begin{proof}
(1)$\Rightarrow$(2):
Let $0\to M^{\oplus a}\to K\to\syz^nM^{\oplus b}\to0$ be an exact sequence with $n\ge0$, $a,b>0$ and $K\in\X$.
Then $M$ belongs to $\red\X$ by definition.
Since it is assumed that $\red\X=\X$, the module $M$ belongs to $\X$.

(2)$\Rightarrow$(1):
Let $M\in\red\X$.
There exists a family $\{0\to K_{i-1}^{\oplus a_i}\to K_i\to\syz^{n_i}K_{i-1}^{\oplus b_i}\to0\}_{i=1}^r$ of exact sequences with $r,n_i\ge0$, $a_i,b_i>0$, $K_0=M$ and $K_r\in\X$.
By assumption and induction on $r$, we get $M\in\X$.
\end{proof}

\begin{dfn}
Let $\X$ be a subcategory of $\mod R$.
\begin{enumerate}[(1)]
\item
We say that $\X$ is {\em closed under subobjects} provided that for an exact sequence $0\to M\to X$ in $\mod R$, if $X$ belongs to $\X$, then $M$ also belongs to $\X$.
\item
We say that $\X$ is {\em closed under kernels} provided that for an exact sequence $0\to M\to X\to Y$ in $\mod R$, if both $X$ and $Y$ belong to $\X$, the module $M$ also belongs to $\X$.
\end{enumerate}
\end{dfn}

We state two sufficient conditions for a given subcategory of $\mod R$ to be reducingly closed.

\begin{prop}
A subcategory $\X$ of $\mod R$ is reducingly closed if either of the following conditions holds.
\begin{enumerate}[\rm(1)]
\item
The subcategory $\X$ is closed under subobjects.
\item
The subcategory $\X$ contains the projective $R$-modules and is closed under finite direct sums and kernels.
\end{enumerate}
\end{prop}

\begin{proof}
Let $0\to M^{\oplus a}\to K\to\syz^nM^{\oplus b}\to0$ be an exact sequence in $\mod R$ with $n\ge0$, $a,b>0$ and $K\in\X$.
In view of Lemma \ref{1}, it suffices to show that $M$ belongs to $\X$.

(1) There are exact sequences $0\to M^{\oplus a}\to K$ and $0\to M\to M^{\oplus a}$.
As $K$ is in $\X$, so is $M^{\oplus a}$ and so is $M$.

(2) There are exact sequences $0\to M^{\oplus a^2}\to K^{\oplus a}\to\syz^nM^{\oplus ab}\to0$ and $0\to\syz^nM^{\oplus ab}\to\syz^nK^{\oplus b}\to\syz^{2n}M^{\oplus b^2}\to0$ with $n\ge0$.
Splicing these exact sequences, we get an exact sequence
\begin{equation}\label{14}
0\to M^{\oplus a^2}\to K^{\oplus a}\to\syz^nK^{\oplus b}.
\end{equation}
As $\X$ is closed under finite direct sums, $K^{\oplus a}$ belongs to $\X$.
There is a series of exact sequences
$$
\{0\to\syz^{i+1}K\to P_i\to\syz^iK\to0\}_{i=1}^{n-1}
$$
of $R$-modules with each $P_i$ projective.
As $\X$ contains all the $P_i$ and is closed under kernels, we inductively see that $\syz^nK$ is in $\X$, and so is $\syz^nK^{\oplus b}$ as $\X$ is closed under finite direct sums.
Since $\X$ is closed under kernels, it is observed from \eqref{14} that $M^{\oplus a^2}$ belongs to $\X$.
Splicing the split short exact sequences $0\to M\to M^{\oplus a^2}\to M^{\oplus a^2-1}\to0$ and $0\to M^{\oplus a^2-1}\to M^{\oplus a^2}\to M\to 0$, we get an exact sequence $0\to M\to M^{\oplus a^2}\to M^{\oplus a^2}$.
Using the assumption that $\X$ is closed under kernels again, we obtain the containment $M\in\X$.
\end{proof}

\section{Reducing projective and Gorenstein dimensions}\label{D}

In this section, we focus on reducing invariants with respect to projective and Gorenstein dimensions. 
First, we give another reducing invariant with respect to ``\,{\it ${}^\ast$reducing}''. 

\begin{dfn}\label{D1}
We define the {\em ${}^\ast$reducing $\I_R$-invariant} of $M$, denoted $\sri_R(M)$, by replacing ``$\red$'' with ``\,$\sred$'' in the definition of the reducing $\I_R$-invariant of $M$. 
That is, we consider a family $\{0\to K_{i-1}\to K_i\to\syz^{n_i}K_{i-1}\to0\}_{i=1}^r$ instead of $\{0\to K_{i-1}^{\oplus a_i}\to K_i\to\syz^{n_i}K_{i-1}^{\oplus b_i}\to0\}_{i=1}^r$. 
\end{dfn}

Note by definition that $\sri_R(M)\le\ri_R(M)$ holds.

\begin{dfn}
Denote by $\fpd R$ the subcategory of $\mod R$ consisting of $R$-modules of finite projective dimension, and by $\fgd R$ the subcategory of $\mod R$ consisting of $R$-modules of finite Gorenstein dimension.
\end{dfn}

\begin{prop}\label{4}
The following equalities hold true.
\begin{align*}
\red(\fpd R)&=\{M\in\mod R\mid\text{$M$ has finite reducing projective dimension}\},\\
\sred(\fpd R)&=\{M\in\mod R\mid\text{$M$ has finite ${}^\ast$reducing projective dimension}\},\\
\red(\fgd R)&=\{M\in\mod R\mid\text{$M$ has finite reducing Gorenstein dimension}\},\\
\sred(\fgd R)&=\{M\in\mod R\mid\text{$M$ has finite ${}^\ast$reducing Gorenstein dimension}\}.
\end{align*}
\end{prop}

\begin{proof}
We only prove the first equality, because the others are shown similarly.
Let $M$ be an $R$-module.

$(\supseteq)$
If $\rpd M=0$, then $M$ has finite projective dimension and we get $M\in\fpd R\subseteq\red(\fpd R)$.
Assume that we have $\rpd M=r$ for some positive integer $r$.
Then there exists a family $\{0\to K_{i-1}^{\oplus a_i}\to K_i\to\syz^{n_i}K_{i-1}^{\oplus b_i}\to0\}_{i=1}^r$ of exact sequences of $R$-modules such that $n_i\ge0$, $a_i,b_i>0$, $K_0=M$ and $\pd_RK_r<\infty$.
The module $K_r$ belongs to $\fpd R$, and we observe that $M$ belongs to $\red(\fpd R)$.

$(\subseteq)$
Suppose that there exists a family $\{0\to K_{i-1}^{\oplus a_i}\to K_i\to\syz^{n_i}K_{i-1}^{\oplus b_i}\to0\}_{i=1}^r$ of exact sequences in $\mod R$ such that $r,n_i\ge0$, $a_i,b_i>0$, $K_0=M$ and $K_r\in\fpd R$.
If $r=0$, then we have $M=K_0=K_r\in\fpd R$, and get $\rpd M=0<\infty$.
If $r>0$, then $\rpd M\le r$ and $M$ has finite reducing projective dimension.
\end{proof}

\begin{cor}
Let $R$ be a singular local complete intersection.
Then $\fpd R$ is not ${}^\ast$reducingly closed.
Hence it is not reducingly closed, either.
\end{cor}

\begin{proof}
The residue field $k$ of $R$ has finite ${}^\ast$reducing projective dimension as an $R$-module by \cite[Corollary 3.8]{red}.
Proposition \ref{4} implies that $k$ belongs to $\sred(\fpd R)$.
Since $R$ is singular, $k$ does not belong to $\fpd R$.
\end{proof}

\begin{ques}
Is there an example of a ring $R$ such that $\fgd R$ is not reducingly closed?
\end{ques}

\begin{dfn}
Suppose that $R$ is a local ring.
We denote by $\fcx(R)$ the subcategory of $\mod R$ consisting of $R$-modules of finite complexity.
\end{dfn}

\begin{cor}
Let $R$ be a local ring.
One then has an inclusion $\sred(\fpd R)\subseteq\fcx(R)$.
\end{cor}

\begin{proof}
The assertion follows from the second equality in Proposition \ref{4} and \cite[Theorem 3.6]{red}.
\end{proof}

\section{Reducingly closed properties}\label{E}

In this section, we show that several concrete subcategories are reducingly closed.
In particular, for certain homological properties in the representation theory of algebras, such as the generalized Auslander--Reiten conjecture, we prove that the full subcategories consisting of modules satisfying these properties are reducingly closed.

\begin{dfn}
Let $R$ be a Cohen--Macaulay local ring.
\begin{enumerate}[(1)]
\item
We denote by $\cm(R)$ the subcategory of $\mod R$ consisting of maximal Cohen--Macaulay $R$-modules.
\item
An $R$-module $M$ is called {\em Ulrich} provided that $M$ is a maximal Cohen--Macaulay $R$-module such that the (Hilbert--Samuel) multiplicity of $M$ is equal to the minimal number of generators of $M$.
We denote by $\ul(R)$ the subcategory of $\mod R$ consisting of Ulrich $R$-modules.
\end{enumerate}
\end{dfn}

\begin{thm}\label{6}
Let $R$ be a Cohen--Macaulay local ring.
Then $\cm(R)$ and $\ul(R)$ are reducingly closed.
\end{thm}

\begin{proof}
We show the assertion on $\cm(R)$ in (1) and that on $\ul(R)$ in (2) below.

(1) Assume that there is a short exact sequence of the form $0\to M^{\oplus a}\to K\to\syz^nM^{\oplus b}\to0$ with $n\ge0$, $a,b>0$ and $K\in\cm(R)$.
Note that $\depth M \le  \depth \Omega^n M$ since $R$ is Cohen--Macaulay. 
Suppose that $M$ is not maximal Cohen--Macaulay.
Then $\depth M < \depth R = \depth K$. 
Applying the depth lemma to the short exact sequence we obtain $\depth \Omega^nM^b = \depth M^a -1 < \depth M$. 
This is a contradiction. 
Therefore, $M$ must be maximal Cohen--Macaulay, that is to say, $M$ belongs to $\cm(R)$.
Applying Lemma \ref{1}, we are done. 

(2) Let $0\to M^{\oplus a}\to K\to\syz^nM^{\oplus b}\to0$ be an exact sequence with $n\ge0$, $a,b>0$ and $K\in\ul(R)$.
Then $K$ is maximal Cohen--Macaulay, and so is $M$ by the argument given in the proof of (1).
It follows from \cite[Lemma 5.5(2)]{umm} that $M$ is Ulrich.
Lemma \ref{1} completes the proof of the assertion.
\end{proof}

\begin{cor}\label{7}
Let $R$ be a Cohen--Macaulay local ring.
Let $\{0\to K_{i-1}^\oplus\to K_i\to\syz^{n_i}K_{i-1}^\oplus\to0\}_{i=1}^r$ be a family of exact sequences.
If $K_r$ is maximal Cohen--Macaulay (resp. Ulrich), then so is $K_0$.
\end{cor}

\begin{proof}
Let $\X$ be either $\cm(R)$ or $\ul(R)$.
If $K_r$ is in $\X$, then $K_0$ is in $\red\X$.
Theorem \ref{6} implies $\X=\red\X$.
\end{proof}

Next we consider the following four properties. 

\begin{dfn}
Let $M$ be an $R$-module.
\begin{enumerate}[(1)]
\item
We say that $M$ satisfies the {\em Auslander condition} provided that there exists an integer $b_M\ge0$ such that for every $R$-module $N$, if $\Ext_R^{\gg0}(M,N)=0$, then $\Ext_R^{>b_M}(M,N)=0$.
\item
We say that $M$ satisfies the {\em uniform Auslander condition} provided that there exists an integer $b_M\ge0$ such that for every nonzero $R$-module $N$, if $\Ext_R^{\gg0}(M,N)=0$, then $\Ext_R^{>b_M}(M,N)=0$ and $\Ext_R^{b_M}(M,N)\neq0$.
\item
We say that $M$ satisfies the {\em generalized Auslander--Reiten conjecture} provided that if there exists an integer $n\ge0$ such that $\Ext_R^{>n}(M,M\oplus R)=0$, then $M$ has projective dimension at most $n$.
\item
We say that $M$ satisfies {\em dependence of the total reflexivity conditions} provided that if $\Ext_R^{>0}(M,R)=0$, then $M$ is totally reflexive.
\end{enumerate}
\end{dfn}

\begin{dfn}
We deal with the following subcategories of $\mod R$:
\begin{align*}
\ac(R)&=\{M\in\mod R\mid\text{$M$ satisfies the Auslander condition}\},\\
\uac(R)&=\{M\in\mod R\mid\text{$M$ satisfies the uniform Auslander condition}\},\\ 
\garc(R)&=\{M\in\mod R\mid\text{$M$ satisfies the generalized Auslander--Reiten conjecture}\},\\
\dtrc(R)&=\{M\in\mod R\mid\text{$M$ satisfies dependence of the total reflexivity conditions}\}.
\end{align*}
\end{dfn}

\begin{thm}\label{2}
One has that $\ac(R)$, $\uac(R)$ and $\garc(R)$ are reducingly closed subcategories of $\mod R$.
\end{thm}

\begin{proof}
Let $0\to M^{\oplus a}\to K\to\syz^nM^{\oplus c}\to0$ be an exact sequence in $\mod R$ with $n\ge0$ and $a,c>0$.
We show the assertions on $\ac(R)$, $\uac(R)$ and $\garc(R)$ in (1), (2) and (3) below, respectively.

(1) Let $K\in\ac(R)$.
There is $b\ge0$ such that $\Ext_R^{>b}(K,H)=0$ for all $R$-modules $H$ with $\Ext_R^{\gg0}(K,H)=0$.

Fix an $R$-module $L$ with $\Ext_R^{\gg0}(M,L)=0$.
The induced exact sequence $\Ext_R^{i+n}(M,L)^{\oplus c}\to\Ext_R^i(K,L)\to\Ext_R^i(M,L)^{\oplus a}$ for each $i>0$ yields $\Ext_R^{\gg0}(K,L)=0$, and so $\Ext_R^{>b}(K,L)=0$.
The induced exact sequence
$$
\Ext_R^i(K,L)\to\Ext_R^i(M,L)^{\oplus a}\to\Ext_R^{i+n+1}(M,L)^{\oplus c}\to\Ext_R^{i+1}(K,L)
$$
for $i>0$ shows $\Ext_R^i(M,L)^{\oplus a}\cong\Ext_R^{i+n+1}(M,L)^{\oplus c}$ for all $i>b$.
We get an isomorphism $\Ext_R^i(M,L)^{\oplus a^l}\cong\Ext_R^{i+l(n+1)}(M,L)^{\oplus c^l}$ for all $i>b$ and $l>0$.
Choosing a big enough integer $l$, we see that $\Ext_R^i(M,L)^{\oplus a^l}=0$ for all $i>b$.
Thus $\Ext_R^{>b}(M,L)=0$ for all $R$-modules $L$ with $\Ext_R^{\gg0}(M,L)=0$.
As $b$ is independent of the choice of $L$, the module $M$ belongs to $\ac(R)$.
Lemma \ref{1} implies that $\ac(R)$ is reducingly closed.

(2) The proof follows almost the same argument of (1). 
Let $K\in\uac(R)$.
Then there exists an integer $b\ge0$ such that $\Ext_R^{>b}(K,H)=0$ and $\Ext_R^{b}(K,H)\neq 0$ for all $R$-modules $H \neq 0$ with $\Ext_R^{\gg0}(K,H)=0$.

Fix an $R$-module $L \neq 0$ with $\Ext_R^{\gg0}(M,L)=0$. 
Then, as in the proof of (1), we have $\Ext_R^{>b}(M,L)=0$ for all $R$-modules $L$ with $\Ext_R^{\gg0}(M,L)=0$.
Assume that $\Ext_R^{b}(M,L)=0$. 
Then $\Ext_R^{b+n}(M,L) = 0$ for $n \ge 0$. 
The exact sequence $\Ext_R^{b+n}(M,L)^{\oplus c}\to\Ext_R^{b}(K,L)\to\Ext_R^{b}(M,L)^{\oplus a}$ shows that $\Ext_R^{b}(K,L) = 0$. 
This is a contradiction, so that $\Ext_R^{b}(M,L)\neq0$. 
Hence, Lemma \ref{1} implies that $\uac(R)$ is reducingly closed.

(3) Let $K\in\garc(R)$.
Suppose that there exists an integer $m\ge0$ such that $\Ext_R^{>m}(M,M\oplus R)=0$.
The induced exact sequence $\Ext_R^{i+n}(M,M\oplus R)^{\oplus c}\to\Ext_R^i(K,M\oplus R)\to\Ext_R^i(M,M\oplus R)^{\oplus a}$ of Ext modules for every integer $i>0$ gives rise to the vanishing $\Ext_R^{>m}(K,M\oplus R)=0$.
Therefore, there are isomorphisms
$$
\Ext_R^i(K,\syz^nM)\cong\Ext_R^{i-1}(K,\syz^{n-1}M)\cong\cdots\cong\Ext_R^{i-n}(K,M)=0
$$
for every $i>m+n$.
From the induced exact sequence $\Ext_R^i(K,M)^{\oplus a}\to\Ext_R^i(K,K)\to\Ext_R^i(K,\syz^nM)^{\oplus c}$ for each $i>0$ we observe that $\Ext_R^{>m+n}(K,K)=0$.
Hence $\Ext_R^{>m+n}(K,K\oplus R)=0$.
Since $K\in\garc(R)$, we have that $\pd K\le m+n$.
Applying the $(m+n)$th syzygies, we get an exact sequence $0\to\syz^{m+n}M^{\oplus a}\to F\to\syz^{m+2n}M^{\oplus c}\to0$ with $F$ free, and we get $(\syz^{m+n}M)^{\oplus a}\cong(\syz^{m+2n+1}M)^{\oplus c}$ up to free summands.
Hence $(\syz^{m+n}M)^{\oplus a^l}\cong(\syz^{m+n+l(n+1)}M)^{\oplus c^l}$ up to free summands for all $l>0$.
Note that there are isomorphisms
$$
\Ext_R^i(M,\syz^jM)\cong\Ext_R^{i-1}(M,\syz^{j-1}M)\cong\cdots\cong\Ext_R^{i-j}(M,M)=0
$$
for all $j\ge0$ and $i>j+m$.
Choose $l$ so that $l(n+1)>m$.
Put $N=\syz^{m+n+l(n+1)-1}M$.
It holds that
$$
\begin{array}{l}
\Ext_R^1(N,\syz N)^{\oplus c^l}
\cong\Ext_R^{m+n+l(n+1)}(M,(\syz^{m+n+l(n+1)}M)^{\oplus c^l})\\
\phantom{\Ext_R^1(N,\syz N)^{\oplus c^l}}\cong\Ext_R^{m+n+l(n+1)}(M,(\syz^{m+n}M)^{\oplus a^l})
\cong\Ext_R^{l(n+1)}(M,M)^{\oplus a^l}=0.
\end{array}
$$
It follows that $\Ext_R^1(N,\syz N)=0$.
This vanishing implies that the exact sequence $0\to\syz N\to G\to N\to0$ with $G$ free splits, so that $N$ is free.
Hence $M$ has finite projective dimension.
Since $\Ext_R^{>m}(M,R)=0$, we see that $\pd_RM\le m$.
Thus $M$ belongs to $\garc(R)$.
Lemma \ref{1} implies that $\garc(R)$ is reducingly closed.
\end{proof}

\begin{dfn}
For two integers $m,n\ge0$ we denote by $\G_{n, m}$ the subcategory of $\mod R$ consisting of $R$-modules $M$ such that $\Ext_R^i(M,R)=0$ for all $1\le i\le n$ and that $\Ext_R^j(\tr M,R)=0$ for all $1\le j\le m$.
We put $\G_{\infty , m}=\bigcap_{i\ge0}\G_{i, m}$, $\G_{n, \infty}=\bigcap_{j\ge0}\G_{n, j}$ and $\G_{\infty , \infty}=\bigcap_{i,j\ge0}\G_{i, j}$.
Note that $\G_{\infty , \infty}$ coincides with the subcategory of $\mod R$ consisting of totally reflexive $R$-modules.
\end{dfn}

The subcategories $\G_{n, m}$ are related to one another.
Indeed, for $M \in \G_{n, m}$, we have $\syz M \in \G_{n-1, m+1}$ if $n>0$, and $\tr M \in \G_{m, n}$.
See \cite[Proposition 1.1.1]{I}.

Now we can prove the following theorem.

\begin{thm}\label{12}
Fix an integer $r\ge0$.
Let $\X$ be the subcategory of $\mod R$ consisting of modules $X$ such that
$$
\Ext_R^{>r}(X,R)=0\implies\gd_RX\le r
$$
holds.
Then $\X$ is a reducingly closed subcategory of $\mod R$.
\end{thm}

\begin{proof}
Assume there is an exact sequence $0\to M^{\oplus a}\to K\to\syz^nM^{\oplus b}\to0$ with $n\ge0$, $a,b>0$ and $K\in\X$.
By Lemma \ref{1}, it suffices to show that $M$ belongs to $\X$.
Taking the $r$th syzygies, we get an exact sequence
\begin{equation}\label{10}
0\to\syz^rM^{\oplus a}\xrightarrow{f}\syz^rK\xrightarrow{g}\syz^{n+r}M^{\oplus b}\to0.
\end{equation}

Let us show that the following implication holds true.
\begin{equation}\label{11}
m\in\Z_{\ge0},\,\syz^rM\in\G_{\infty , m}\implies \syz^rM\in\G_{\infty,m+n+1}.
\end{equation}
Suppose that $m\in\Z_{\ge0}$ and $\syz^rM\in\G_{\infty , m}$ hold true.
Then the syzygy $\syz^{n+r}M$ belongs to $\G_{\infty,m+n}$.
Applying the functor $\Ext_R^i(-,R)$ for $i>0$ to the exact sequence \eqref{10}, we see that $\syz^rK$ belongs to $\G_{\infty , 0}$.
Hence $\Ext_R^{>r}(K,R)=0$.
Since $K$ is in $\X$, we have $\gd_RK\le r$.
Therefore $\syz^rK$ is totally reflexive, that is, $\syz^rK\in\G_{\infty , \infty}$.
Set $(-)^\ast=\Hom_R(-,R)$.
By \cite[Lemma 3.9]{AB}, there is an exact sequence
$$
0\to(\syz^{n+r}M^{\oplus b})^\ast\xrightarrow{g^\ast}(\syz^rK)^\ast\xrightarrow{f^\ast}(\syz^rM^{\oplus a})^\ast\to\tr\syz^{n+r}M^{\oplus b}\xrightarrow{\tr g}\tr \syz^rK\xrightarrow{\tr f}\tr\syz^rM^{\oplus a}\to0.
$$
There is an exact sequence $(\syz^rK)^\ast\xrightarrow{f^\ast}(\syz^rM^{\oplus a})^\ast\to\Ext_R^1(\syz^{n+r}M^{\oplus b},R)$ and the last Ext module is isomorphic to $\Ext_R^{n+1}(\syz^rM,R)^{\oplus b}$, which vanishes.
We get a short exact sequence
$$
0\to\tr\syz^{n+r}M^{\oplus b}\to\tr\syz^rK\to\tr\syz^rM^{\oplus a}\to0.
$$
We have that $\tr\syz^{n+r}M\in\G_{m+n,\infty}$, that $\tr \syz^rK\in\G_{\infty , \infty}$ and that $\tr\syz^rM\in\G_{m , \infty}$.
For each $2\le i\le m+n+1$ it holds that $\Ext_R^i(\tr \syz^rM^{\oplus a},R)\cong\Ext_R^{i-1}(\tr\syz^{n+r}M^{\oplus b},R)=0$.
As $\syz^rK$ is totally reflexive, it is torsionless.
It is seen from \eqref{10} that $\syz^rM$ is torsionless as well, i.e., $\syz^rM$ is in $\G_{0, 1}$.
Hence $\tr\syz^rM$ is in $\G_{1, 0}$.
Thus $\Ext_R^i(\tr\syz^rM,R)=0$ for all $1\le i\le m+n+1$, that is, $\tr\syz^rM\in\G_{m+n+1,0}$.
We get $\syz^rM\in\G_{0,m+n+1}$ and obtain the containment $\syz^rM\in\G_{\infty , m}\cap\G_{0,m+n+1}=\G_{\infty,m+n+1}$.
The proof of \eqref{11} is completed.

Now, suppose that $\Ext_R^{>r}(M,R)=0$.
Then $\syz^rM$ belongs to $\G_{\infty , 0}$.
Applying \eqref{11} repeatedly, we observe that $\syz^rM$ belongs to $\G_{\infty,l(n+1)}$ for all integers $l\ge0$.
This means that $\syz^rM$ is in $\G_{\infty , \infty}$, i.e., $\syz^rM$ is totally reflexive.
Consequently, we obtain the inequality $\gd_RM\le r$.
We now conclude that $M$ belongs to $\X$.
\end{proof}

Applying Theorem \ref{12} to $r=0$, we get the following corollary.

\begin{cor}\label{9}
The subcategory $\dtrc(R)$ of $\mod R$ is reducingly closed.
\end{cor}

\begin{cor}\label{8}
Let $\P$ be one of the following properties:
\begin{itemize}
\item
the Auslander condition,
\item
the uniform Auslander condition,
\item
the generalized Auslander--Reiten conjecture, 
\item
dependence of the total reflexivity conditions.
\end{itemize}
Let $\{0\to K_{i-1}^{\oplus a_i} \to K_i\to\syz^{n_i}K_{i-1}^{\oplus b_i}\to0\}_{i=1}^r$ be a family of exact sequences.
If $K_r$ satisfies $\P$, so does $K_0$.
\end{cor}

\begin{proof}
Let $\X$ be the subcategory of $\mod R$ consisting of $R$-modules that satisfy the property $\P$.
The module $K_r$ belongs to $\X$, so that $K_0$ belongs to $\red\X$.
It follows from Theorem \ref{2} and Corollary \ref{9} that $\X=\red\X$.
Consequently, the module $K_0$ satisfies the property $\P$.
\end{proof}

The corollary below is an immediate consequence of Corollaries \ref{7} and \ref{8}.

\begin{cor}\label{13}
Let $\I_R :(\mod R/\cong)\to\Z\cup\{\pm\infty\}$ be a map.
\begin{enumerate}[\rm(1)]
\item
Let $\P$ be the Auslander condition, 
uniform Auslander condition, 
the generalized Auslander--Reiten conjecture or dependence of total reflexivity conditions.
Suppose that all $R$-modules $M$ such that $\I_R (M)<\infty$ satisfy the property $\P$.
Then all $R$-modules $M$ such that $\ri _R (M)<\infty$ satisfy the property $\P$.
\item
Suppose that $R$ is a Cohen--Macaulay local ring.
Let $\P$ be maximal Cohen--Macaulayness or Ulrichness.
If all $R$-modules $M$ with $\I_R(M)<\infty$ satisfy $\P$, then so are all $R$-modules $M$ with $\ri _R (M)<\infty$.
\end{enumerate}
\end{cor}

At the end of the paper, we list some classes of modules that we can apply Corollary \ref{13}.

\begin{rem}
For an $R$-module $M$, there is the following hierarchy of homological dimensions: 
$$\pd M \geq \CIdim M \geq \ABdim M  \geq \Gdim M.$$
See \cite[Theorem 1.2]{A17} for instance. 
Therefore we also have the following: 
$$
\rpd M \geq \rCIdim M \geq \rABdim M  \geq \rGdim M. 
$$
\end{rem}

Notice that we have inclusions of subcategories $\uac (R) \subseteq \ac (R) \subseteq \garc (R)$.  
See \cite[Theorem 1.2]{CT}. 

\begin{cor}\label{E1}
Let $M$ be an $R$-module. 
If $\rCIdim M < \infty$, then $M$ satisfies the uniform Auslander condition. 
Consequently, both the Auslander--Reiten conjecture and the generalized Auslander--Reiten conjecture hold. 
Moreover $M$ satisfies dependence of the total reflexivity conditions. 
\end{cor}

\begin{proof}
Let $X$ be of finite projective dimension. 
Then it is clear that $X$ satisfies the uniform Auslander condition. 
Indeed $\Ext_R ^{> \pd X} (X, N)=0$ for each $R$-module $N$. 
Thus $\fpd (R)$ is contained in $\uac(R)$, so that $\fpd (R) \subseteq \uac (R) \subseteq \ac (R) \subseteq \garc (R)$. 
It is also clear that $X$ satisfies dependence of the total reflexivity conditions. 
Because if $\Ext _R ^{>0} (X, R) = 0$ then $X$ must be projective,  which implies that $X$ is totally reflexive. 

By Corollary \ref{13}, if $\rpd M < \infty$ then $M$ satisfies the uniform Auslander condition, the generalized Auslander--Reiten conjecture and the generalized Auslander--Reiten conjecture. 
Let $M$ be an $R$-module with $\rCIdim M < \infty$. 
Then $\rpd M < \infty$ (\cite[2.7]{CDKM22}). 
Applying the above observation, we are done. 
\end{proof}

\begin{cor}\label{E2}
Let $M$ be an $R$-module.   
\begin{enumerate}[\rm(1)]

\item Suppose that $\rABdim M < \infty$. 
Then $M$ satisfies the generalized Auslander--Reiten conjecture. 
 
\item Suppose that $\rGdim M < \infty$. 
Then $M$ satisfies dependence of the total reflexivity conditions. 

\end{enumerate}
\end{cor}

\begin{proof}
(1) Let $\fab (R)$ be a subcategory of $\mod R$ consisting of $R$-modules of finite AB dimension. 
Let $X \in \fab (R)$ and assume that $\hExt (X, X \oplus R) = 0$. 
In particular, $\hExt (X, X) = 0$. 
Then $\pd X < \infty$ (cf. \cite[Corollary 2.8]{A17}).  
Thus $\fab (R)$ satisfies the generalized Auslander--Reiten conjecture.  
By Corollary \ref{13}, if $\rABdim M < \infty$ and $\hExt(M, M \oplus R) = 0$, then $\pd M < \infty$.

(2) Let $X \in \fgd (R)$, that is $\Gdim X < \infty$. 
Suppose that $\Ext_R ^{>0} (X, R) = 0$. 
Then $\Gdim X = 0$, so that $X$ is totally reflexive. 
Hence $ \fgd (R)$ satisfies dependence of the total reflexivity conditions.
By Corollary \ref{13}, we obtain the assertion.  
\end{proof}

By the definition, if $M$ has reducible complexity, then $M$ has finite reducing projective dimension. 
Therefore we have the following. 

\begin{cor}\label{E3}
Let $M$ be an $R$-module with reducible complexity. 
Then $M$ satisfies the (uniform) Auslander condition, the (generalized) Auslander--Reiten conjecture and dependence of the total reflexivity conditions.
\end{cor}



\end{document}